\documentclass{amsart}

\usepackage[active]{srcltx}

\usepackage{amsmath,amssymb,amsthm}
\newtheorem{theorem}{Theorem}[section]
\newtheorem{lemma}[theorem]{Lemma}
\newtheorem{proposition}[theorem]{Proposition}
\newtheorem{corollary}[theorem]{Corollary}
\newtheorem{definition}[theorem]{Definition}
\newtheorem{remark}[theorem]{Remark}
\newtheorem{example}[theorem]{Example}

\begin{document}

\title[Indecomposable injective modules of finite Malcev rank]{Indecomposable injective modules of finite Malcev rank over local commutative rings}
\author{Fran\c{c}ois Couchot}
\address{Universit\'e de Caen Basse-Normandie, CNRS UMR
  6139 LMNO,
F-14032 Caen, France}
\email{francois.couchot@unicaen.fr} 


\begin{abstract}
It is proven that each indecomposable injective module over a valuation domain $R$ is polyserial if and only if each maximal immediate extension  $\widehat{R}$ of $R$ is of finite rank over the completion $\widetilde{R}$ of $R$ in the $R$-topology. In this case, for each indecomposable injective module $E$, the following invariants are finite and equal: its Malcev rank, its Fleischer rank and its dual Goldie dimension. Similar results are obtained for chain rings satisfying  some additional properties. It is also shown that each indecomposable injective module over  local Noetherian rings of Krull dimension one has finite Malcev rank. The preservation of the finiteness of Goldie dimension  by localization is investigated too.
\end{abstract}

\maketitle

\section*{Introduction and preliminaries}
In this paper all rings are associative and commutative with unity and
all mo\-dules are unital. First we give some definitions. 

\bigskip

\begin{definition}
An $R$-module $M$ is said to be \textbf{uniserial} if its set of submodules is totally ordered by inclusion and  $R$ is a \textbf{chain ring}\footnote{we prefer ``chain ring '' to ``valuation ring'' to avoid confusion with ``Manis valuation ring''.} if it is uniserial as $R$-module. A chain domain is a valuation domain. In the sequel, if $R$ is a chain ring, we denote by $P$ its maximal ideal, $N$ its nilradical, $Z$ its set of zero-divisors ($Z$ is a prime ideal) and we put $Q=R_Z$. Recall that a chain ring $R$ is said to be {\bf Archimedean} if $P$ is the sole non-zero prime ideal.

A module $M$ is said to be \textbf{finitely cogenerated} if its injective hull is a finite direct sum of injective hulls of simple modules. The \textbf{f.c. topology} on a module $M$ is the linear topology defined by taking as a basis of neighbourhoods of zero all submodules $G$ for which $M/G$ is finitely cogenerated (see \cite{Vam75}). This topology is always Hausdorff. We denote by $\widetilde{M}$ the completion of $M$ in its f.c. topology. When $R$ is a chain ring which is not a finitely cogenerated $R$-module, the f.c. topology on $R$ coincides with the $R$-topology which is defined by taking as a basis of neighbourhoods of zero all non-zero principal ideals.
A chain ring $R$ is said to be \textbf{(almost) maximal} if $R/A$ is complete in its f.c. topology for any (non-zero) proper ideal $A$.
\end{definition}

\bigskip

In 1959, Matlis proved that a valuation domain $R$ is almost maximal if and only if $Q/R$ is injective, and in this case, for each proper ideal $A$ of $R$, $E(R/A)\cong Q/A$, see \cite[Theorem 4]{Mat59}. Since $Q$ is clearly uniserial and $Q/R\cong Q/rR$ for each non-zero element $r\in P$, we can also say that $R$ is almost maximal if and only if $E(R/rR)$ is uniserial, if and only if each indecomposable injective module is uniserial. This result was extended to any chain ring in 1971 by Gill, see \cite[Theorem]{Gil71}: a chain ring $R$ is almost maximal if and only if $E(R/P)$ is uniserial, if and only if each indecomposable injective module is uniserial. By using \cite[Proposition 14]{Couch03}, if $R$ is a chain ring, it is easy to check that $E(R/rR)$ is uniserial if and only if so is $E(R/P)$. Let us observe that any indecomposable injective module is uniserial if and only if each finitely generated uniform module is cyclic.

\bigskip

\begin{definition}
If $M$ is a finitely generated module we denote by $\mathrm{gen}\ M$ its
minimal number of generators. If $M$ is a module over a valuation domain $R$ the {\bf Fleischer rank} of $M$, denoted by $\mathrm{Fr}\ M$, is defined to be the minimum rank of torsion-free modules having $M$ as an epimorphic image. 
\end{definition}

\bigskip

In the book ``Modules over valuation domains'' by Fuchs and Salce \cite[Proposition IX.3.1]{FuSa85}(1985), it is proven that $\mathrm{gen}\ M\leq\mathrm{Fr}\ E(R/P)$, for each  finitely generated uniform module $M$ over a valuation domain $R$. However, it remains to give a characterization of valuation domains $R$ for which $\mathrm{Fr}\ E(R/P)$ is finite.

In 2005 \cite[Proposition 2]{Cou05}, if $R$ is an Archimedean chain ring, the author proved that there exists an integer $p>0$ such that $\mathrm{gen}\ M\leq p$ for each finitely generated uniform module $M$ if and only if $R$ is almost maximal, (i.e $p=1$).

\bigskip

\begin{definition}
An exact sequence \ $0 \rightarrow F \rightarrow E \rightarrow G \rightarrow 0$ \ is \textbf{pure}
if it remains exact when tensoring it with any $R$-module. In this case
we say that \ $F$ \ is a \textbf{pure} submodule of $E$. We say that a module $M$ is \textbf{polyserial} if it has a
pure-composition series
\[0=M_0\subset M_1\subset\dots\subset M_n=M,\]
i.e. $M_k$ is a pure submodule of $M$
and $M_k/M_{k-1}$ is a uniserial module for each $k=1,\dots, n$. If the submodules $M_k$ are no longer to be assumed pure in $M$, then we say $M$ \textbf{weakly polyserial}.

The \textbf{Malcev rank} of a module $M$ is defined as the cardinal
number
\[\mathrm{Mr}\ M=\mathrm{sup}\{\mathrm{gen}\ X\mid X\ \mathrm{finitely\ generated\ submodule\ of}\ M\}.\] 
For each module $M$ over a valuation domain we have $\mathrm{Mr}\ M\leq\mathrm{Fr}\ M$.

An $R$-module $F$ is \textbf{pure-injective} if for every pure exact
sequence
\[0\rightarrow A\rightarrow B\rightarrow C\rightarrow 0\]
 of $R$-modules, the following sequence
 \[0\rightarrow\mathrm{Hom}_R(C,F)
\rightarrow\mathrm{Hom}_R(B,F)\rightarrow\mathrm{Hom}_R(A,F)\rightarrow
0\] is exact. An $R$-module $B$
is a \textbf{pure-essential extension} of a submodule $A$ if $A$ is a
pure submodule of $B$ and, if for each
submodule $K$ of $B$, either $K\cap A\ne 0$ or $(A+K)/K$ is not a pure
submodule of $B/K$. 
We say that $B$ is a \textbf{pure-injective hull}
of $A$ if $B$ is pure-injective and a pure-essential extension of $A$.
 By \cite{War69} or \cite[chapter XIII]{FuSa01} each $R$-module $M$ has a
 pure-injective hull and any two pure-injective hulls of $M$ are isomorphic.
In the sequel, for each $R$-module $M$, 
$\widehat{M}$ is its pure-injective hull.
\end{definition}

\bigskip

In this paper we  give a characterization of two classes of chain rings. The first is the class of chain rings $R$ for which each indecomposable injective module is polyserial (Theorem~\ref{T:main}). These rings are exactly the chain rings $R$  which satisfies the following two conditions:
\begin{enumerate}
\item $\mathrm{Mr}_{\widetilde{R}}\widehat{R}<\infty$ \footnote{we shall see that $\widehat{R}$ can be viewed as a $\widetilde{R}$-module by Proposition~\ref{P:compl}(2)};
\item each indecomposable injective module contains a pure uniserial submodule\footnote{This condition holds for each valuation domain and other classes of chain rings but  we don't know if it is verified by any chain ring.}.
\end{enumerate}
The first condition holds if and only if any indecomposable injective module is weakly polyserial. It is also equivalent to the the following: there is a non-zero prime ideal $L$ such that $R_L$ is almost maximal and the valuation domain $R/L$ has a maximal immediate extension of finite rank which is equal to $\mathrm{Mr}_{\widetilde{R}}\widehat{R}$.
These rings  are almost maximal by stages, i.e. there exists a finite descending chain  of prime ideals $(L_i)_{0\leq i\leq n}$, with $L_0=P$  such that $(R/L_{i+1})_{L_i}$ is almost maximal for $i=0,\dots,n-1$ and $R_{L_n}$ is  maximal.  Moreover, for each finitely generated uniform module $M$, $\mathrm{gen}\ M\leq\mathrm{Mr}_{\widetilde{R}}\widehat{R}$, and for each indecomposable injective module $E$, $\mathrm{Mr}_RE\leq\mathrm{Mr}_{\widetilde{R}}\widehat{R}$, the equalities hold for some $M$ and $E$. If $R$ is not a domain then $\mathrm{Mr}\ \widehat{R}<\infty$. A description of such chain rings $R$   is given, and  this description is similar to the one  of valuation domains with a maximal immediate extension of finite rank (\cite[Theorem 10 and Proposition 11]{Cou10}).

The second class is the one of chain rings $R$ for which each localization of any $R$-module of finite Goldie dimension has finite Goldie dimension too. These rings are exactly the chain rings $R$ for which $R/L$ has a maximal immediate extension of finite rank for each non-zero prime ideal $L$. So, the first class is contained but strictly in the second  one, and some examples are given.

It is also shown that the completion $\widetilde{R}$ of any chain ring $R$ in its f.c. topology is Gaussian, and $\widetilde{R}$ is a chain ring if and only if $R$ is either complete or a domain.

For each local Noetherian ring of Krull dimension one $R$ it is proven that there exists a positive integer $m$ such that $\mathrm{Mr}\ E\leq m$ for every indecomposable injective $R$-module $E$. Moreover, for each integer $m>1$ we give an example of a local Noetherian domain of Krull dimension one $D$ satisfying $\mathrm{gen}\ M\leq m$ for each finitely generated uniform $D$-module $M$. However, if $R$ is a chain ring with  such an upper bound $m$  then $m$ is a prime power.

\bigskip

\begin{definition}\label{D:sharp}
 Let $M$ be a non-zero module over a  ring $R$. 
 We set:
\[M_{\sharp}=\{s\in R\mid \exists 0\ne x\in M\ \mathrm{such\ that}\
sx=0\}\quad\mathrm{and}\quad M^{\sharp}=\{s\in R\mid sM\subset M\}.\] 
Then $R\setminus M_{\sharp}$ and $R\setminus M^{\sharp}$ are multiplicative subsets of $R$. 

If $M$ is a module over a chain ring $R$ then $M_{\sharp}$ and $M^{\sharp}$ are prime ideals and they are called the {\bf bottom} and the {\bf top prime ideal}, respectively, associated with $M$.

We say that an $R$-module
$E$ is  \textbf{FP-injective} if $\mathrm{Ext}_R^1(F,E)=0,$ for every finitely
presented $R$-module $F.$ A ring $R$ is called \textbf{self
FP-injective} if it is FP-injective as $R$-module.  Recall that a
module $E$ is FP-injective if and only if it is a pure submodule
of every overmodule.

If $L$ is a prime ideal of a chain ring $R$, as in \cite{FaZa86}, we define the \textbf{total defect} at $L$, $d_R(L)$, the \textbf{completion defect} at $L$, $c_R(L)$, as the rank of the torsion-free $R/L$-module $\widehat{R/L}$ and the rank of the torsion-free $R/L$-module $\widetilde{R/L}$, respectively. 
\end{definition}

\bigskip

\section{Relations  between  $\widehat{R}$ and $\widetilde{R}$}
\label{S:compl}

Given a ring $R$, an $R$-module $M$ and $x\in M$,  the \textbf{content ideal} $\mathrm{c}(x)$ of $x$ in $M$, is the intersection of all ideals $A$ for which $x\in AM$. 

When $R$ is a chain ring, the \textbf{breadth ideal} $\mathrm{B}(x)$ of an element $x$ in $\widehat{R}$ is defined by $\mathrm{B}(x)=\mathrm{c}(x+R)$ ($x+R\in \widehat{R}/R$). So, $\mathrm{B}(x)=0$ if $x\in R$. Since $\widehat{R}=R+P\widehat{R}$ by \cite[Proposition 1]{Couc06} then $\mathrm{B}(x)=\{r\in R\mid x\notin R+r\widehat{R}\}$ if $x\in\widehat{R}\setminus R$.

\medskip

The following lemma will be often used in the sequel.
\begin{lemma}
\label{L:breadth}\cite[Proposition 20 and Lemma 21]{Couc06} Let $R$ be a chain ring. Then:
\begin{enumerate}
\item $R/A$ is not complete in its f.c. topology if and only if $A=\mathrm{B}(x)$ for some $x\in\widehat{R}\setminus R$;
\item if 
  $x=r+ay$ where $r,a\in R$ and $x,y\in\widehat{R}\setminus R$, then
  $\mathrm{B}(y)=(\mathrm{B}(x):a)$.
\end{enumerate}

\end{lemma}

\begin{proposition}\label{P:compl} Let $R$ be a chain ring. Then:
\begin{enumerate}
\item $\widetilde{R}$ is a local ring;
\item $\widehat{R}$ has a structure of $\widetilde{R}$-module which extends its structure of $R$-module;
\item $\widetilde{R}$ is   isomorphic  to the submodule of $\widehat{R}$ whose elements $x$ satisfy $\mathrm{B}(x)=0$;
\item for each non-zero prime ideal $L$ of $R$ there exists a prime ideal $L'$ of $\widetilde{R}$ such that $\widetilde{R}/L'\cong R/L$ and $L'\widehat{R}=L\widehat{R}$;
\item $\widetilde{R}/R$ is a  $Q/Z$-vector space;
\item each element $a$ of $\widetilde{R}\cap(1+P\widehat{R})$ is inversible.
\end{enumerate}
\end{proposition}
\begin{proof} 
$(1)$.   $\widetilde{R}$ is local because it is the inverse limit of a system of local rings with local connecting homomorphisms. 

$(2)$. If $R$ is finitely cogenerated then $\widetilde{R}=R$. If not we have $\cap_{r\in R\setminus \{0\}}rP=0$.
Let $a\in\widetilde{R}$ and $x\in\widehat{R}$. Let $(a_r+rP)_{r\in R\setminus \{0\}}$ be the family of cosets of $R$ which defines $a$. If $r\in sR$ then $(a_r-a_s)\in sP$, and it follows that $(a_rx-a_sx)\in sP\widehat{R}$. By \cite[Proposition 4]{Couc06} the family $\mathcal{F}=(a_rx+rP\widehat{R})_{r\in R\setminus \{0\}}$ has a non-empty intersection. By \cite[Lemma 19]{Couc06} $\cap_{r\in R\setminus \{0\}}rP\widehat{R}=0$, whence the intersection of the  family $\mathcal{F}$ contains a unique element that we define to be $ax$. Now it is easy to complete the proof.

$(3)$. We do as in $(2)$ by taking $x=1$. So, for each $a\in \widetilde{R}$ corresponds a unique element $y\in\widehat{R}$ such that $\mathrm{B}(y)=0$. It is easy to check that we get a monomorphism from $\widetilde{R}$ into $\widehat{R}$.

$(4)$. We may assume that $R$ is not finitely cogenerated. Since each non-zero ideal is open in the f.c. topology of $R$, we have $\widetilde{R}\cong\varprojlim_{A\in\mathcal{I}}R/A$, where $\mathcal{I}$ is the set of non-zero ideals of $R$. So, there exists a surjection $\phi:\widetilde{R}\rightarrow R/L$. We put $L'=\ker\ \phi$. Let $0\ne a\in L'$ and $0\ne x\in\widehat{R}$, and let $(a_r+rR)_{r\in R\setminus \{0\}}$ be the family of cosets of $R$ which defines $a$. There exists $r\in L$ such that $a_r\in L\setminus rR$. We set $a'=a_r$. Let $s\in rR$. Since $(a_s-a')\in rR\subset a'R$ then $a_s=a'u_s$ for some $u_s\in R$. If $t\in sR$ then $a'(u_s-u_t)\in sR$, whence $(u_s-u_t)\in (sR:a')$. The family $(u_sx+(sR:a')\widehat{R})_{s\in Ra'\setminus \{0\}}$ has a non-empty intersection. Let $y$ be an element of this intersection. Since $(a'y-a_sx)\in s\widehat{R}$ for each $s\in a'R\setminus \{0\}$, it follows that $(a'y-ax)\in\cap_{s\in a'R\setminus \{0\}}s\widehat{R}=0$. Hence $ax=a'y$.

$(5)$. Let $x\in\widetilde{R}$ and $s\in R\setminus Z$. Suppose that $sx=0$. By \cite[Proposition 1]{Couc06} $x=r+rpy$ for some $r\in R$, $p\in P$ and $y\in\widehat{R}$. Since $R$ is a pure submodule of $\widehat{R}$, there exists $t\in R$ such that $sr(1+pt)=0$. We successively deduce that $sr=0$, $r=0$ and $x=0$. Now, suppose that $sx\in R$. Then there exists $t\in R$ such that $sx=st$, whence $x=t$. So, the multiplication by $s$ in $\widetilde{R}/R$ is injective. Since $\mathrm{B}(x)=0$, $x=a+sy$ for some $a\in R$ and $y\in\widehat{R}$. But $\mathrm{B}(y)=(0:s)=0$, so $y\in\widetilde{R}$. We conclude that the multiplication by $s$ in $\widetilde{R}/R$ is bijective. Now let $a\in Z$. Then $(0:a)$ contains a non-zero element $b$. From $\mathrm{B}(x)=0$ we deduce that $x=r+bz$ for some $r\in R$ and $z\in\widehat{R}$. It follows that $ax\in R$.

$(6)$. We have $a=1+px$ for some $p\in P$ and $x\in\widehat{R}$. First suppose  $Z\ne P$. We may assume that $p\notin Z$. So, by Lemma~\ref{L:breadth} $\mathrm{B}(x)=(0:p)=0$, whence $x\in\widetilde{R}$ and $a$ is a unit since $\widetilde{R}$ is local. Now, suppose $Z=P$ and let $0\ne t\in (0:p)\cap Rp$. From $\mathrm{B}(a)=0$, we deduce that $a=u+ty$ for some $u\in R$ and $y\in\widehat{R}$. It follows that $(u-1)\in p\widehat{R}\cap R=pR$ ($R$ pure submodule of $\widehat{R}$). Hence $u$ is a unit. We have $a(u^{-1}-u^{-2}ty)=1-u^{-2}(ty)^2$. By using $(2)$, $(ty)^2=t(ty)y=(t^2y)y=0$. Hence $a$ is a unit.
\end{proof}

A local ring $R$ is called \textbf{Gaussian}\footnote{this definition is equivalent to the usual one when $R$ is local, see \cite{Tsa65}.} if, for any ideal $A$ generated by two elements $a,b,$ in $R$, the   following two properties hold:
\begin{enumerate}
\item $A^2$ is generated by  $a^2$ or $b^2$;
\item if $A^2$ is generated by $a^2$ and $ab=0$, then $b^2=0$.
\end{enumerate}

\begin{theorem}
\label{T:compl} Let $R$ be a  chain ring. The following assertions hold:
\begin{enumerate}
\item $\widetilde{R}$ is a local Gaussian ring;
\item the following conditions are equivalent:
\begin{enumerate}
\item $\widetilde{R}$ is a chain ring;
\item $R$ is either complete or  a domain;
\item $\widetilde{R}$ is a pure $R$-submodule of $\widehat{R}$;
\item $\widetilde{R}$ is a flat $R$-module.
\end{enumerate}
\end{enumerate}
\end{theorem}
\begin{proof}
$(1)$.
We may assume that $R$ is not finitely cogenerated. Let $a$ and $b$ be two elements of $\widetilde{R}$. By \cite[Proposition 1]{Couc06} there exist $a',b'\in R$ and $x,y\in 1+P\widehat{R}$ such that $a=a'x$ and $b=b'y$. We  may assume  that $b'=ra'$ for some $r\in R$. First suppose that $a'\notin Z$. By Lemma~\ref{L:breadth}(2) $\mathrm{B}(x)=(0:a')=0$, whence $x\in\widetilde{R}$, and $\mathrm{B}(y)=(0:ra')=(0:r)$. By Proposition~\ref{P:compl}  $x$ is a unit. Since $\mathrm{B}(ry)=0$,  $ry\in\widetilde{R}$, so $b=x^{-1}(ry)a$. If $ab=0$, it follows that $(ry)a^2=0$, whence $b^2=0$. Now, assume that $0\ne a'\in Z$. Let $0\ne t\in(0:a')\cap b'P$. Then $a=c+tz$ and $b=d+tw$ for some $c,d\in R$ and $z,w\in\widehat{R}$. So, $c-a'\in a'P\widehat{R}\cap R=a'PR$ ($R$ pure submodule of $\widehat{R}$), whence  $c=ua'$ for some unit $u\in R$. In the same way  $d=vb'$ for some unit $v\in R$. Since $a't=0$, it follows that $a^2=u^2a'^2$, $ab=uvra'^2=u^{-1}vra^2$ and $b^2=v^2r^2a'^2=u^{-2}r^2v^2a^2$. If $ab=0$, then $ra'^2=0$, whence $b^2=0$.

$(2)$. It is well known that $(b)$ implies the other three conditions.

Assume that $R$ is neither complete nor a domain. Let $a\in\widetilde{R}\setminus R$ and $0\ne r\in Z$. Since $\mathrm{B}(a)=0$ then $a=s+rx$ for  some $s\in R$ and $x\in\widehat{R}$. So $rx\in\widetilde{R}\setminus R$. If $rx=ry$ then $y\notin R$  and $\mathrm{B}(y)=(0:r)\ne 0$ by Lemma~\ref{L:breadth}(2), whence $rx\notin r\widetilde{R}$ and $\widetilde{R}$ is not a pure submodule of $\widehat{R}$. Now suppose that $r\in\widetilde{R}rx$, whence $r=brx$ for some $b\in\widetilde{R}$.  From $r(bx-1)=0$ and the flatness of $\widehat{R}$ we deduce that $bx-1=sy$ for some $s\in(0:r)$ and $y\in\widehat{R}$. Then $b$ is a unit, else, from Proposition~\ref{P:compl}(4) we get that $1\in P\widehat{R}\cap R=P$. It follows that $rx=rb^{-1}\in r\widetilde{R}$. This is false. So, $\widetilde{R}$ is not a chain ring. Hence, $(a)\Rightarrow (b)$ and $(c)\Rightarrow (b)$.

$(d)\Rightarrow (b)$. Since $R$ is a pure submodule of $\widetilde{R}$, $\widetilde{R}/R$ is flat. By Proposition~\ref{P:compl} it is a semisimple $Q$-module. It follows that either $\widetilde{R}/R=0$ or $Q$ is a field. We conclude that the  condition $(b)$ holds.
\end{proof}

\begin{proposition}
\label{P:complMr} Let $R$ be a chain ring. Assume   $(0:a)=Z$ for some $a\in R$. Then $\mathrm{Mr}\ \widetilde{R}=c_R(Z)$.
\end{proposition}
\begin{proof}
Let $R'=R/Z$. We shall prove that $\widetilde{R'}/R'$ and $\widetilde{R}/R$ are isomorphic. Since $\widetilde{R'}\subseteq\widehat{R'}\cong\widehat{R}/Z\widehat{R}$, $\widetilde{R'}/R'$ is isomorphic to the submodule of $\widehat{R}/(R+Z\widehat{R})$ whose elements $x+(R+Z\widehat{R})$ satisfy $\mathrm{B}(x)=Z$ (it is easy to check that $\mathrm{B}(x')=\mathrm{B}(x)$ if $x'\in x+(R+Z\widehat{R})$ and $x\notin R+Z\widehat{R}$). On the other hand, $\widetilde{R}\subseteq R+a\widehat{R}$, whence $\widetilde{R}/R$ is isomorphic to a submodule of $(R+a\widehat{R})/R$. For each $x\in\widehat{R}$ we put $\phi(x+(R+Z\widehat{R}))=ax+R$. It is easy to check that $\phi$ is a well defined epimorphism from $\widehat{R}/(R+Z\widehat{R})$ into $(R+a\widehat{R})/R$. If $ax\in R$, then $ax=ad$ for some $d\in R$ because $R$ is a pure submodule of $\widehat{R}$. From $a(x-d)=0$ and the flatness of $\widehat{R}$ we deduce that $(x-d)\in Z\widehat{R}$. So, $\phi$ is an isomorphism. By Lemma~\ref{L:breadth}(2) $\mathrm{B}(x)=Z$ if and only if $\mathrm{B}(ax)=0$. Consequently, the restrition of $\phi$ to $\widetilde{R'}/R'$ is an isomorphism onto $\widetilde{R}/R$.
\end{proof}

\bigskip

\section{Polyserial injective modules}
\label{S:poly}

 The following proposition is a slight generalization of \cite[Proposition 2]{Cou05}.
\begin{proposition} \label{P:Arch}
Let $R$ be an Archimedean chain ring. Assume that there exists a non-zero injective module $E$ such that $E_{\sharp}=P$ and $\mathrm{Mr}\ E<\infty$. Then $R$ is almost maximal.
\end{proposition}
\begin{proof}
It is an immediate consequence of \cite[Proposition 2]{Cou05} and its proof. The existence of an injective module $E$ with $\mathrm{Mr}\ E<\infty$ ($\nu(E)<\infty$)\footnote{the Malcev rank of $E$ is denoted by $\nu(E)$ in \cite{Cou05}} is used to show that $R$ is almost maximal.
\end{proof}

\begin{lemma}\label{L:max} Let $R$ be a maximal chain ring and let $M$ be a flat module such that $M/PM$ is finitely generated. Then $M$ is a free module of rank  $\mathrm{gen}\ M/PM$.
\end{lemma}
\begin{proof}
Let $p=\mathrm{gen}\ M/PM$. By \cite[Proposition 21]{Coucho07} $M$ contains a pure-essential free submodule $F$ of rank $p$. Since $R$ is maximal $F$ is pure-injective. So, $F=M$.
\end{proof}

We say that a module $M$ is \textbf{singly projective} if, for any  cyclic submodule $G$, the inclusion map $G\rightarrow M$ factors through a free module $F$. The following theorem generalizes \cite[Theorem 10 and Proposition 11]{Cou10}

\begin{theorem}
\label{T:InjHull} Let $R$ be a chain ring. The following conditions are equivalent:
\begin{enumerate}
\item $\widehat{R}$ is a polyserial module;
\item $\mathrm{Mr}\ \widehat{R}<\infty$.
\end{enumerate}
In this case there exists a finite family of prime ideals \[P=L_0\supset L_1\supset\dots\supset L_{m-1}\supset L_m\supseteq N\] such that   $(R/L_{k+1})_{L_k}$ is almost maximal, $\forall k,\ 0\leq k\leq m-1$, and $R_{L_m}$ is maximal. 

Moreover,  
\begin{itemize}
\item[(a)] $\widehat{R}$ has a pure-composition series $\mathcal{S}$
\[0=F_0\subset R=F_1\subset\dots\subset F_m\subset F_{m+1}=\widehat{R}\]
where $F_{j+1}/F_j$ is a free $R_{L_j}$-module of finite rank, $\forall j,\ 0\leq j\leq m$;
\item[(b)] $\mathrm{Mr}\ \widehat{R}=d_R(L_m)=\prod_{k=1}^{k=m}c_R(L_k)$;
\item[(c)] $\widetilde{R}$ is polyserial and $\mathrm{Mr}\ \widetilde{R}=c_R(Z)$. 
\end{itemize}
\end{theorem}
\begin{proof}
By \cite[Lemma XII.1.4]{FuSa01} $(1)\Rightarrow (2)$.

$(2)\Rightarrow (1)$. When $R$ is a valuation domain each torsion-free module of finite rank is polyserial.  So, we may assume that $R$ is not a domain. 

First we will show that $R_N$ is maximal. Since each non-unit of $R_N$ is a zero-divisor, $R_N$ is self FP-injective. From \cite[Proposition 1]{Couc06} it is easy to deduce that $\widehat{R}$ is  singly projective. By \cite[Proposition 6]{Coucho07}  $(\widehat{R})_N$ singly projective over $R_N$. By \cite[Proposition  3]{Coucho07} $(\widehat{R})_N$ is FP-injective and by \cite[Proposition 5]{Couc06} it is pure-injective, whence it is an injective module. It is easy to check that $\mathrm{Mr}\ (\widehat{R})_N<\infty$. By proposition~\ref{P:Arch} and \cite[Proposition 1]{Gil71} we conclude that $R_N$ is maximal. 

Now we shall build the pure composition series $\mathcal{S}$. By \cite[Theorem 2.4.(2)]{Couc06} $\widehat{R/N}=\widehat{R}/N\widehat{R}$. Hence $\mathrm{rank}\ \widehat{R/N}=\mathrm{Mr}\ \widehat{R/N}<\infty$. We apply \cite[Theorem 10 and Proposition 11]{Cou10} to $R/N$. There exists a finite family of prime ideals \[P=L_0\supset L_1\supset\dots\supset L_{m-1}\supset L_m\supseteq N\] such that   $(R/L_{k+1})_{L_k}$ is almost maximal, $\forall k,\ 0\leq k\leq m-1$, and $(R/N)_{L_m}$ is maximal. Moreover, $\widehat{R/N}$ has a pure-composition series $\mathcal{S'}$
\[0=F'_0\subset R/N=F'_1\subset\dots\subset F'_m\subset F'_{m+1}=\widehat{R/N}\]
where $F'_{j+1}/F'_j$ is a free $(R/N)_{L_j}$-module of finite rank, $\forall j,\ 0\leq j\leq m$. We proceed by induction on $j$. Obviously $F_1=R$. Suppose that $F_j$ is built and that $F_j/NF_j\cong F'_j$. If $M=\widehat{R}/F_j$ and $M'=\widehat{R/N}/F'_j$ then $M/NM\cong M'$. So, $M^{\sharp}=L_j$, whence $M$ is a module over $R_{L_j}$. Moreover,  $M/L_jM$ and $M'/L_jM'$ have the same rank $p_j$ over $(R/L_j)_{L_j}$ which is equal to the rank of $F'_{j+1}/F'_j$ over $(R/N)_{L_j}$. By \cite[Proposition 21]{Coucho07} $M$ contains a pure free $R_{L_j}$-submodule $G$ of rank $p_j$. Moreover, $G/NG\cong F'_{j+1}/F'_j$ and $L_j(M/G)=M/G$. Let $F_{j+1}$ be the inverse image of $G$ by the natural map $\widehat{R}\rightarrow M$. Hence $F_{j+1}/NF_{j+1}\cong F'_{j+1}$. Now, let $H=\widehat{R}/F_{m+1}$. Thus $H$ is flat and $NH=H$ because $F_{m+1}'=\widehat{R}/N\widehat{R}$. By \cite[Proposition 19]{Coucho07} $H$ is a module over $R_N$. It is obvious that $(F_{m+1})_N$ is a free $R_N$-module of finite rank equal to $\mathrm{gen}\ (\widehat{R})_N/N(\widehat{R})_N$. By Lemma~\ref{L:max} $(F_{m+1})_N=(\widehat{R})_N$ . So, $H=H_N=0$ and $F_{m+1}=\widehat{R}$. The maximality of $(R/N)_{L_m}$ and $R_N$ implies that $R_{L_m}$ is maximal if $L_m\ne N$ (see \cite[Theorem 22]{Couc06}).

(b). We apply the last assertion of \cite[Theorem 10]{Cou10} to $R/N$.

(c). We have $\mathrm{Mr}\ \widetilde{R}\leq\mathrm{Mr}\ \widehat{R}<\infty$. So, by Proposition~\ref{P:compl} $\widetilde{R}/R$ is a finite direct sum of modules isomorphic to $Q/Z$, whence $\widetilde{R}$ is polyserial. If $A$ is a non-zero proper ideal it is easy to check that $A^{\sharp}=\{s\in R\mid A\subset (A:s)\}$. So, if we take this definition of top prime ideal for each proper ideal of $R$ we have $0^{\sharp}=Z$, and for each $0\ne t\in R$, $(0:t)^{\sharp}=Z$.  We shall show that  there exists $a\in Z$ such that $Z=(0:a)$.  First assume that $N\subset Z$. If $t\in Z\setminus N$, then $0\ne (0:t)\subset N$. It follows that $N\subset Rt\subseteq (0:s)$ for some $0\ne s\in(0:t)$. By Lemma~\ref{L:breadth}(1) there exists $x\in\widehat{R}\setminus R$ such that $\mathrm{B}(x)=0$ if $R$ is not complete. Then $x=r+sy$ for some $r\in R$ and $y\in\widehat{R}$, and $\mathrm{B}(y)=(0:s)$. Consequently, by using again Lemma~\ref{L:breadth}(1)  we deduce that $R/(0:s)$ is not complete. If $R'$ is a valuation domain with $\widehat{R'}$ of finite rank, then, by \cite[Theorem 10 and Proposition 11]{Cou10} and their proofs, $R'/A$ is not complete if and only if $A$ is a proper ideal isomorphic to $A^{\sharp}$ and this prime ideal is one of the list $(L_i)$. We apply this result to $R/N$, and  we get that $(0:s)/N=bZ/N$ for some $b\in R\setminus N$. It follows that $(0:s)=bZ$, whence $Z=(0:bs)$. Now, suppose that $Z=N$. If $Z$ is faithful, then $\cap_{a\ne 0}aZ=0$, so, since $R_N$ is maximal, as in the proof of \cite[Proposition 4]{Cou10} we prove that $R$ is complete. Hence, if $R$ is not complete, $Z=(0:a)$ for some $a\in Z$. We conclude by Proposition~\ref{P:complMr}.
\end{proof}

For each module $M$ we denote by $\mathcal{A}(M)$ its set of annihilator ideals, i.e. an ideal $A$ belongs to $\mathcal{A}(M)$ if there exists $0\ne x\in M$ such that $A=(0:x)$. If $E$ is an indecomposable injective module over a chain ring $R$, then, for any $A,\ B\in\mathcal{A}(E),\ A\subset B$ there exists $r\in R$ such that $A=rB$ and $B=(A:r)$.

Recall that a module $M$ has \textbf{Goldie dimension} $n$ (or $\mathrm{Gd}\ M=n$) if its injective  hull is a  direct sum of $n$ indecomposable injective modules.

\begin{theorem}
\label{T:main} Let $R$ be a chain ring. Consider the following conditions:
\begin{enumerate}
\item  there exists an indecomposable injective module $E$ such that $E_{\sharp}=P$ and $\mathrm{Mr}\ E<\infty$;
\item  there exists a  prime ideal $L$ such that $d_R(L)<\infty$  and  $R_L$ is almost maximal;
\item  $\mathrm{Mr}\ \widehat{R}<\infty$ if $R$ is not a domain, and $\widehat{R}$ is the extension of a reduced torsion-free module of finite rank with a divisible torsion-free module when $R$ is a domain;
\item the Malcev rank of $\widehat{R}$ over $\widetilde{R}$ is finite;
\item $\mathrm{Mr}\ E<\infty$ for each indecomposable injective module $E$;
\item there exists a positive integer $n$ such that $\mathrm{gen}\ M\leq n$ for each  finitely generated uniform $R$-module $M$;
\item there exists a positive integer $n$ such that $\mathrm{gen}\ M\leq n\mathrm{Gd}\ M$ for each  finitely generated  $R$-module $M$;
\item each indecomposable injective module is weakly polyserial;
\item there exists an indecomposable injective module $E$ such that $E_{\sharp}=P$ which is weakly polyserial.
\item each indecomposable injective module is polyserial;
\item there exists an indecomposable injective module $E$ such that $E_{\sharp}=P$ which is polyserial.
\end{enumerate}
Then:
\begin{itemize}
\item[(a)]  the   first nine conditions are equivalent and they are implied by the  last two conditions. Moreover, if each indecomposable injective module contains a pure uni\-serial submodule then  the eleven conditions are equivalent.
\item[(b)]  for each indecomposable injective module $E$, either $\mathrm{Mr}\ E=1$ if $J\subseteq L$ or $\mathrm{Mr}\ E=d_{R_J}(L_J)=d_R(L)/d_R(J)=\mathrm{Mr}_{\widetilde{R_J}}\widehat{R_J}$ if $L\subset J$, where $J=E_{\sharp}$ and $L$ is a prime ideal for which $R_L$ is almost maximal. Moreover, $\mathrm{Mr}\ E$ is the maximum of $\mathrm{gen}_R\ M$ where $M$ runs over all finitely generated $R$-submodules of  uniform $R_J$-modules.
\end{itemize}
\end{theorem}
\begin{proof}
(a). It is obvious that $(8)\Rightarrow (9)$, $(10)\Rightarrow (11)$, $(11)\Rightarrow (9)$, $(6)\Rightarrow (5)$, $(7)\Rightarrow (6)$ and $(5)\Rightarrow (1)$, and    $(9)\Rightarrow (1)$ by \cite[Corollary XII.1.5]{FuSa01}. 

$(1)\Rightarrow (2)$. By \cite[Corollary 28]{Couch03}  $E$ is faithful or it is annihilated by a simple ideal if $P=Z$. So, for each non-zero prime ideal $J$ there exists $A\in\mathcal{A}(E)$ such that $A\subset J$.  By \cite[Lemma 26]{Couch03} $A^{\sharp}=E_{\sharp}=P$ and by \cite[Proposition 1]{Couc06} $\widehat{R}/A\widehat{R}$ is an essential extension of $R/A$, whence it is isomorphic to a  submodule of $E$. We   deduce that $d_R(J)=\mathrm{Mr}\ \widehat{R}/J\widehat{R}\leq\mathrm{Mr}\ \widehat{R}/A\widehat{R}\leq\mathrm{Mr}\ E<\infty$.  Let $p$ be the maximum of $d_R(J)$ where $J$ runs over all non-zero prime ideals of $R$ and let $L$ be the maximal prime ideal for which $d_R(L)=p$. By Theorem~\ref{T:InjHull} $(R/I)_L$ is maximal for each $I\in\mathcal{A}(E)$. We deduce that $R_L$ is almost maximal. Let us observe that $d_R(L)\leq \mathrm{Mr}\ E$.

$(2)\Rightarrow (3)$ and $(4)$. We do as in the proof of Theorem~\ref{T:InjHull}: from a pure composition series of $R/L$ we deduce a pure submodule $F$ of $\widehat{R}$ with $\mathrm{Mr}\ F=d_R(L)$ and if $H=\widehat{R}/F$, then $LH=H$. 
If $R$ is not a domain then, 
as in the proof of Theorem~\ref{T:InjHull}, we show that $H=0$, whence $\widehat{R}$ is polyserial.  It is easy to check that $\mathrm{Mr}_{\widetilde{R}}\widehat{R}\leq\mathrm{Mr}_R\widehat{R}$. If $R$ is a domain then, for each $a\in L$, $a\ne 0$, $(R/aR)_L$ is maximal. In the same way we get that $H=aH$, whence $H$ is divisible. 
Since $\widetilde{R}/L\widetilde{R}=R/L$   then $d_{\widetilde{R}}(L\widetilde{R})=d_R(L)$ and $(\widetilde{R})_L$ is maximal. So, $\widehat{R}$ is the extension of a torsion-free $\widetilde{R}$-module $F$ of rank $d_R(L)$ with a divisible torsion-free $\widetilde{R}$-module $H$. By Lemma~\ref{L:max} $F_L=(\widehat{R})_L$ because $F_L$ is free over $(\widetilde{R})_L$ and $(\widetilde{R})_L$ is maximal. Hence $H=H_L=0$ and $F=\widehat{R}$. So, $\mathrm{Mr}_{\widetilde{R}}\ \widehat{R} = d_R(L)$.

$(4)\Rightarrow (2)$.  Let $J$ be a non-zero prime ideal of $R$. By Proposition~\ref{P:compl} there exists a  prime ideal $J'$ of $\widetilde{R}$ such that  $\widetilde{R}/J'=R/J$ and $J'\widehat{R}=J\widehat{R}$. So, $\mathrm{Mr}_R\ \widehat{R/J}\leq \mathrm{Mr}_{\widetilde{R}}\ \widehat{R}$. Now, we do as in $(1)\Rightarrow (2)$ to complete the proof.

$(3)\Rightarrow (6)$. Let $M$ be a  finitely generated uniform module and  $E$ its injective hull. Then $E$ is  indecomposable. Let $J=E_{\sharp}$. Then $E$ is a module over $R_J$. If $J\subseteq L$ then $E$ is uniserial, so $\mathrm{Mr}\ E=1$. We may assume that $L\subset J$ and it is easy to check that $R_J$ also satisfies $(3)$.  There exists $A\in\mathcal{A}(E)$ such that $M\subseteq (0:_EA)$. Let $E'=(0:_EA)$. By \cite[Lemma 26]{Couch03} $A^{\sharp}=J$, so, by \cite[Theorem 6]{Couc06} $E'\cong\widehat{R_J}/A\widehat{R_J}$. If $R$ is  a domain, let $F$ be a pure reduced torsion-free $R_J$-submodule of finite rank of $\widehat{R_J}$  such that $\widehat{R_J}/F$ is divisible. Then $(\widehat{R_J}/F)\otimes_{R_J}(R/A)_J=0$. So, in this case $E'\cong (F/AF)_J$ and $\mathrm{Mr}\ E'\leq \mathrm{rank}\ F$. If $R$ is not a domain, then $\mathrm{Mr}\ E'\leq\mathrm{Mr}\ \widehat{R_J}$. We deduce that  $\mathrm{gen}\ M\leq \mathrm{rank}\ F$ or $\mathrm{gen}\ M\leq \mathrm{Mr}\ \widehat{R_J}$.

$(6)\Rightarrow (7)$. Let $M$ be a  finitely generated  module and  $E$ its injective hull. By \cite[Corollary IX.2.2]{FuSa85} $\mathrm{Gd}\ M\leq\mathrm{gen}\ M$. So, $E=\oplus_{1\leq j\leq p}E_j$ where $E_j$ is indecomposable for $j=1,\dots,p$. Let $\pi_j: E\rightarrow E_j$ be the natural projection and $M_j=\pi_j(M)$. Then $M$ is isomorphic to a submodule of $\oplus_{1\leq j\leq p}M_j$. By $(6)$ $\mathrm{gen}\ (\oplus_{1\leq j\leq p}M_j)\leq np$. Since each finitely generated ideal is principal, we conclude that $\mathrm{gen}\ M\leq np$ by \cite[Lemma 1.3]{WiWi75}.

$(3)\Rightarrow (10)$. Let $E$ be an indecomposable injective module. We assume that $E$ contains a pure uniserial submodule $U$. If $J=E_{\sharp}$, then $E\cong\widehat{R_J}\otimes_{R}U$ by \cite[corollary 11.(4)]{Couc06}. If $R$ is not a domain, then $\widehat{R_J}$ is polyserial by Theorem~\ref{T:InjHull}. If $R$ is a domain, we may assume that $J\ne 0$. Let $F$ be a pure reduced torsion-free $R$-submodule of finite rank of $\widehat{R_J}$ such that $\widehat{R_J}/F$ is divisible. Then $(\widehat{R_J}/F)\otimes_{R}U=0$, $E\cong F\otimes_RU$ and we know that $F$ is polyserial. So, in the two cases, from a pure composition series of $\widehat{R_J}$ or $F$ with uniserial factors, we deduce a pure composition series of $E$ with uniserial factors. Hence $E$ is polyserial.

$(2)\Rightarrow (8)$. Let $E$ be an indecomposable injective module and $J=E_{\sharp}$. If $J\subseteq L$ then $E$ is a module over $R_L$. So, $E$ is uniserial since $R_L$ is almost maximal. Now, assume that $L\subset J$. We denote by $K$ and $I$ the kernel and the image of the natural map $E\rightarrow E_L$. Since $L\subset J$, then $L\notin\mathcal{A}(E)$. So, $K\cong\mathrm{Hom}_R(R/L,E)$, whence $K$ is an indecomposable injective module over $R/L$. Since $(2)\Rightarrow (3)\Rightarrow (10)$ and any injective module over a valuation domain contains a pure uniserial module, we get that $K$ is polyserial. On the other hand, by using Theorem~\ref{T:main2} in Section~\ref{S:Gol} and the fact that $R_L$ is almost maximal, we deduce that $I$ is a submodule of a finite direct sum of uniserial modules. By \cite[Theorem IX.5.5]{FuSa85}  $I$ is polyserial too. Hence $E$ is weakly polyserial.

(b). The second assertion is also proven.
\end{proof}

\begin{lemma}
\label{L:pureunis} Let $R$ be a chain ring and let $E$ be an indecomposable injective module such that $Z\subset E_{\sharp}$. Assume that $E$ contains a pure uniserial submodule. Then each indecomposable injective module $G$ for which $Z\subset G_{\sharp}$ contains a pure uniserial submodule.
\end{lemma}
\begin{proof}
After replacing $R$ by $R_{E_{\sharp}}$ we may assume that $E_{\sharp}=P$. First we shall prove that $E(R/Z)$ contains a pure uniserial submodule. If $Q$ is coherent, it is a consequence of \cite[Corollary 22]{Couch03}. We assume that $Q$ is not coherent. So, $Z$ is flat by \cite[Theorem 10]{Couch03}. By \cite[Theorem 3]{Cou06} $E_Z$ is injective, $E_Z\ne 0$, and it contains a pure uniserial submodule $U$ and an injective hull of $U$. Let $A\in\mathcal{A}(E)$, $A\subset Z$. Since $A^{\sharp}=E_{\sharp}=P$, there exists $s\in P\setminus Z$ such that $A\subset (A:s)$. Let $t\in (A:s)\setminus A$. Then $Z\subset (A:t)$. So, $(A:t)_Z=Q$. It follows that $A_Z=tQ$. Hence $E(U)\cong E(Q/tQ)$. From \cite[Proposition 14]{Couch03}, we deduce that $E(R/Z)$ contains a pure uniserial submodule $V$. Let $x\in E(R/Z)$ such that $Z=(0:x)$. If $G$ is an indecomposable injective module such that $Z\subset G_{\sharp}$, $\mathcal{A}(G)$ contains a faithful ideal $B$. By \cite[Proposition 6]{Couch03} $V/Bx$ is a pure uniserial submodule of $G$. 
\end{proof}

Let us observe that the condition $(10)$ of Theorem~\ref{T:main} implies that each indecomposable injective module contains a pure uniserial submodule.

\begin{proposition}\label{P:polys}
Let $R$ be a chain ring and let $E$ be an indecomposable injective module such that $P=E_{\sharp}$. Assume that $E$ is polyserial. Then:
\begin{enumerate}
\item each indecomposable injective module $G$ for which $Z\subset G_{\sharp}$ is polyserial;
\item for each prime ideal $L\subseteq Z$, $E(R/L)$ and $E(R_L/aR_L)$ are polyserial, where $a\in L$ with $0\ne aR_L$.
\end{enumerate}
\end{proposition}
\begin{proof}
$(1)$ holds by Lemma~\ref{L:pureunis} and  Theorem~\ref{T:main}.

$(2)$ holds by \cite[Corollary 22]{Couch03} and Theorem~\ref{T:main}.
\end{proof}

\begin{remark}
If $R$ is a chain ring which is not a domain, satisfying $\mathrm{Mr}_{R}\widehat{R}<\infty$, then $\mathrm{Mr}_{R}\widehat{R}=\mathrm{Mr}_{\widetilde{R}}\widehat{R}$ even if $R\subset\widetilde{R}$. 
\end{remark}

\bigskip

\section{Fleischer rank and dual Goldie dimension of indecomposable injective modules}
\label{S:FldG}

\begin{remark}\label{R:Malcev}
If $M$ is a torsion-free module of finite rank over a valuation domain, it is easy to check that its Malcev rank is equal to its rank. So, if $M$ is a  module over a chain ring $R$, then $\mathrm{Fr}\ M$ can be defined to be the minimum Malcev rank of flat modules having $M$ as an epimorphic image.  Obviously $\mathrm{Mr}\ M\leq\mathrm{Fr}\ M$ for each module $M$.
\end{remark}

\begin{proposition} \label{P:IF} Let $R$ be a chain ring and let $E$ be an indecomposable injective module such that $E_{\sharp}\subseteq Z$. Then $E$ is flat if $\mathcal{A}(E)\ne\{qQ\mid 0\ne q\in Z\}$.
\end{proposition}
\begin{proof} If $\mathcal{A}(E)=\{rZ\mid r\in R\}$ then $E$ is flat by \cite[Proposition 8]{Couch03}. So, we may assume that $A$ is not of the form $rZ$ if $A\in\mathcal{A}(E)$. By \cite[Lemma 26]{Couch03} $A^{\sharp}=E_{\sharp}$ for each $A\in\mathcal{A}(E)$, so $A$ is an ideal of $Q$. It is easy to check that $(0:I)$ is also an ideal of $Q$ for each ideal $I$ of $R$. In the sequel we apply \cite[Proposition 1.3]{KlLe69} to $Q$: $(0:(0:A))\ne A$ if and only if $A=qZ$ and $(0:(0:A))=qQ$ for some $q\in Z$.
Let $r\in R$ and $x\in E$ such that $rx=0$. Then $r\in A$ where $A=(0:x)$. Since $rQ\subset A$, then $(0:A)\subset (0:r)$. Let $a\in (0:r)\setminus (0:A)$. It follows that $(0:a)\subseteq (0:(0:A))=A$. The injectivity of $E$ implies that there exists $y\in E$ such that $x=ay$. So, $E$ is flat.
\end{proof}

\begin{proposition}
Let $R$ be a chain ring. Assume that $d_R(L)<\infty$ and $R_L$ is almost maximal for a non-zero prime ideal $L$, and that $E(R/Z)$ contains a pure uniserial submodule $U$. Then $\mathrm{Mr}\ E=\mathrm{Fr}\ E$ for each indecomposable injective module $E$.
\end{proposition}
\begin{proof}
Let $E$ be an indecomposable injective module and $J=E_{\sharp}$. Since $\mathrm{Mr}\ E\leq\mathrm{Fr}\ E$ it is enough to show that $E$ is an epimorphic image of a flat module $G$ with $\mathrm{Mr}\ E=\mathrm{Mr}\ G$. First we assume that $J\subseteq Z$. If $Q$ is coherent then $E$ is flat. If $Q$ is not coherent and if $E\ncong E(Q/qQ)$, where $0\ne q\in Z$, then  $E$ is flat by Proposition~\ref{P:IF}. If $E= E(Q/qQ)$, then by \cite[Proposition 14]{Couch03} there exits an epimorphism $E(Q/Z)\rightarrow E$ whose kernel is a simple $Q$-module. It is easy to check that $\mathrm{Mr}\ E=\mathrm{Mr}\ E(Q/Z)$, and  $E(Q/Z)$ is flat. Now, we assume that $Z\subset J$. In this case, $E\cong\widehat{R_J}\otimes_R(U/Ax)$ where $A$ is a faithful annihilator ideal of $E$ and $x\in U$ with $Z=(0:x)$. Moreover, $U$ is flat because so is $E(R/Z)$. Hence $E$ is an epimorphic image of $\widehat{R_J}\otimes_RU$ which is flat. If $R$ is not a domain then $\mathrm{Mr}\ E=\mathrm{Mr}\ \widehat{R_J}\otimes_RU=\mathrm{Mr}\ \widehat{R_J}$ by Theorem~\ref{T:main}(b). If $R$ is a domain, by Theorem~\ref{T:main} $\widehat{R_J}$ contains a pure submodule $F$ of  rank equal to $\mathrm{Mr}_{\widetilde{R}}\ \widehat{R_J}$ such that $\widehat{R_J}/F$ is a divisible module. In this case we take $U=Q$ and $x=1$. Since $Q/A$ is a torsion module we have $E\cong F\otimes_RQ/A$. So, $E$ is a homomorphic image of $F\otimes_RQ$ and $\mathrm{Mr}\ E=\mathrm{Mr}_{\widetilde{R}}\ \widehat{R_J}=\mathrm{Mr}\ F\otimes_RQ=\mathrm{d}_{R_J}(L)$ by Theorem~\ref{T:main}(b).
\end{proof}

We say that a submodule $K$ of a module $M$ is \textbf{superfluous} if the equality $K+G=M$ holds only when $G=M$. A module $M$ is \textbf{co-uniform} if each of its proper submodules  is superfluous. We say that $M$ has \textbf{dual Goldie dimension} $n$ (or $\mathrm{dG}\ M=n$) if there exists an epimorphism $\phi$ from $M$ into a direct sum of $n$ co-uniform modules such that $\ker\ \phi$ is superfluous.

\begin{proposition}
\label{P:MrdG} Let $R$ be a chain ring. Then $\mathrm{dG}\ M\leq\mathrm{Mr}\ M$ for each $R$-module $M$.
\end{proposition}
\begin{proof}
Let $n$ a positive integer such that $n\leq\mathrm{dG}\ M$. Then there exists an epimorphism $\phi:M\rightarrow \oplus_{i=1}^n M_i$ where $M_i$ is a non-zero $R$-module for $i=1,\dots,n$. For each $i$, $1\leq i\leq n$, let $x_i$ be a non-zero element of $M_i$ and let $y_i\in M$ such that $x_i=\phi(y_i)$. If $\Sigma_{i=1}^na_iy_i=0$ where $a_i\in R$ for $i=1,\dots,n$, we successively deduce that $\Sigma_{i=1}^na_ix_i=0$, $a_ix_i=0$ and $a_i\in P$ for $i=1,\dots,n$. It follows that $\mathrm{Mr}\ M\geq n$ for each integer $n\leq\mathrm{dG}\ M$. So, $\mathrm{dG}\ M\leq\mathrm{Mr}\ M$.
\end{proof}

\begin{proposition}
\label{P:dG} Let $R$ be a chain ring. Suppose there exists a non-zero prime ideal $L$ such that $Z\subseteq L$, $d_R(L)<\infty$ and $R_L$ is almost maximal. Then  each indecomposable injective module $E$ is polyserial and $\mathrm{Mr}\ E=\mathrm{dG}\ E$.
\end{proposition}
\begin{proof}
Let $H$ be the injective hull of $R/Z$. Then, since $H$ is an $R_Z$-module and $R_Z$ is almost maximal, $H$ is uniserial. Let $E$ be an indecomposable injective module and let $J=E_{\sharp}$. If $J\subseteq Z$ then $E$ is uniserial. If $Z\subset J$ we do as in the proof of Lemma~\ref{L:pureunis} to show that $E$ contains a pure uniserial submodule $U$. By Theorem~\ref{T:main} $E$ is a polyserial module.
If $V$ is a uniserial factor of a pure composition series of $\widehat{R_J}$, then by Theorem~\ref{T:InjHull} $V\cong R_{L'}$ for some prime ideal $L'\supseteq L$. Il follows that $E_L\cong U_L^d$ where $d=d_{R_J}(L_J)=\mathrm{Mr}\ E$. Since $Z\subseteq L$, $E_L$ is a homomorphic image of $E$. So, $\mathrm{dG}\ E\geq d$. By Proposition~\ref{P:MrdG} $\mathrm{dG}\ E=d$.
\end{proof}

We say that a chain ring is \textbf{strongly discrete} if $L^2\ne L$ for each non-zero prime ideal $L$.
\begin{proposition}
\label{P:discrete} Let $R$ be a  chain ring such that $Q$ is strongly discrete. Then each indecomposable injective module contains a pure uniserial submodule. For such a ring the eleven conditions of Theorem~\ref{T:main} are equivalent.
\end{proposition}
\begin{proof}
Let $E$ be an indecomposable injective module and let $J=E_{\sharp}$. Then $E$ is an $R_J$-module. If $A\in\mathcal{A}(E)$ then $A^{\sharp}=J$ by \cite[Lemma 26]{Couch03}. If $J\subseteq Z$,  since $J\ne J^2$ then $JR_J=cR_J$ for some $c\in J\setminus J^2$, and since $AR_J\ne cAR_J$ then $AR_J=aR_J$ for some $a\in AR_J\setminus cAR_J$. On the other hand, since $Z=sQ$ for some $s\in Z\setminus Z^2$, $(0:Z)=(0:s)\ne 0$. So, \cite[Corollary 22]{Couch03} can be applied to show that $E$ contains a pure uniserial submodule.
\end{proof}

\begin{corollary}
\label{C:main} Let $R$ be a valuation domain. Then the eleven conditions of Theorem~\ref{T:main} are equivalent. Moreover, for each indecomposable injective module $E$, $\mathrm{Mr}\ E=\mathrm{Fr}\ E=\mathrm{dG}\ E=\mathrm{Mr}_{\widetilde{R_J}}\ \widehat{R_J}$ where $J=E_{\sharp}$.
\end{corollary}

\begin{example}
It is possible to build examples of chain rings satisfying the eleven equivalent conditions of Theorem~\ref{T:main} by using \cite[Example 6 and Theorem 8]{FaZa86}. These examples are strongly discrete (and Henselian). If $R$ is  such  an example then $\mathrm{Mr}_{\widetilde{R}}\ \widehat{R}=p^m$, where $p$ is a prime number and $m$ a non-negative integer. By \cite[Remark p.16]{Vam90} $\mathrm{Mr}_{\widetilde{R}}\ \widehat{R}$ is always a prime power.
\end{example}

\bigskip

\section{Indecomposable injective modules over local Noetherian rings of Krull dimension one}
\label{S:Noe}

From a result by Marie-Paule Malliavin  we deduce Theorem~\ref{T:Noe}. If $M$ is a module of finite length, we denote by $\ell(M)$ its length.

\begin{theorem}
\label{T:Noe} Let $R$ be a local Noetherian ring of Krull dimension one at most. There exists a positive integer $n$ such that $\mathrm{Mr}\ E\leq n$ for each indecomposable injective $R$-module $E$.
\end{theorem}
\begin{proof}
By \cite[Th\'eor\`eme 1.4.2]{Mal66} $\mathrm{Mr}\ R$ is finite. We put $m=\mathrm{Mr}\ R$. Let $E$ be an indecomposable injective module. Then there exists a prime ideal $L$ such that $E=E(R/L)$. First we assume that $L$ is a minimal prime. It follows that $E$ is a module of finite length over $R_L$ by \cite[Theorem 3.11(2)]{Mat58} since $R_L$ is Artinian. In this case $E$ has a composition series whose factors are isomorphic to $R_L/LR_L$. It is easy to see that $\mathrm{Mr}\ R_L/LR_L=\mathrm{Mr}\ R/L\leq m$. Now, by induction on the length of $E$ over $R_L$ and by using \cite[Lemma XII.1.4]{FuSa01} we get that $\mathrm{Mr}\ E<\infty$. Now we assume that $L=P$ the maximal ideal of $R$. Let $M$ be a finitely generated submodule of $E$ and $A=(0:M)$. Since $E$ is Artinian then $M$ is a module of finite length and $R/A$ is Artinian. By \cite[Proposition 1.2]{Cou81} $M$ is injective over $R/A$ and $R/A=\mathrm{End}_R(M)$. Let $B$ be the ideal of $R$ such that $B/A$ is the socle of $R/A$. If $p=\ell(B/A)$ then $p=\mathrm{gen}\ B\leq m$. So, there is an exact sequence $0\rightarrow R/A\rightarrow M^p$ and by applying the functor $\mathrm{Hom}_{R/A}(-,M)$ to this sequence, we get that $M$ is a homomorphic image of $(R/A)^p$. So, $\mathrm{Mr}\ E\leq m$. Since the set of prime ideals of $R$ is finite the theorem is proven.
\end{proof}

\begin{example}
Let $R$ be a local ring of maximal ideal $P$ such that $P^2=0$. If $\mathrm{gen}\ P=n$ where $n>0$ it is easy to check that $\mathrm{Mr}\ E(R/P)\leq n$.
\end{example}

In the sequel, for each integer $n>1$ we shall give an example of a local Noetherian domain $D$ of  Krull dimension one all  whose finitely generated uniform mo\-dules are generated by at most $n$ elements.

\begin{example}
Consider the  Noetherian domain $R$ defined in in the following way. Let $K$ be a field, $K[X,Y]$ the polynomial ring in two variables $X$ and
$Y,$ and $f = Y^n-X^n(1+X).$ By
considering that $f$ is a polynomial in one variable $Y$ with
coefficients in $K[X],$ it follows from Eisenstein's criterion that $f$ is
irreducible. Then $D = \displaystyle{\frac{K[X,Y]}{fK[X,Y]}}$ is a
domain. Let $x$ and $y$ be the  images of $X$ and $Y$ in $D$ by the natural
map and $P'$ the maximal ideal of $D$ generated by $\{x,y\}.$ Let $R=D_{P'}$ and $P=P'R$. 

Then
$\mathrm{Mr}\ R= n$ and $\mathrm{Mr}\ E= n$ for each indecomposable injective $R$-module $E$.
\end{example}
\begin{proof}  There are only two types of indecomposable injective modules: $E(R/P)$ and $Q$ the quotient field of $D$ and $R$.  Let $M$ be a finitely generated submodule of $Q$. Then $M$ is isomorphic to an ideal of $R$. As a module over over $K[X]$, $D$ is generated by $n$ elements $1,y,y^2,\dots,y^{n-1}$. Since $K[X]$ is a principal ideal domain, by \cite[Lemma 1.3]{WiWi75} each $K[X]$-submodule of $D$ is generated by at most $n$ elements. It follows that each ideal of $D$ and each ideal of $R$ is generated by at most $n$ elements. Let us observe that $\mathrm{gen}\ P^m=n$ for each $m\geq n-1$. So, $\mathrm{Mr}\ R=\mathrm{Mr}\ Q=n$. As in the proof of Theorem~\ref{T:Noe} we show that $\mathrm{Mr}\ E(R/P)\leq n$, and since $\mathrm{gen}\ P^{n-1}=n$ we have $\mathrm{gen}\ \mathrm{Hom}_R(R/P^n,E(R/P))=n$. The proof is now complete.
\end{proof}

\bigskip

\section{Goldie dimension and localization}
\label{S:Gol}

At the beginning of this section $R$ is not necessarily a chain ring.

\begin{proposition}
\label{P:Gol} Let $R$ be a ring satisfying one of the following two conditions:
\begin{enumerate}
\item $R_P$ is a domain of Krull dimension one for each maximal ideal $P$;
\item $R_P$ is Noetherian for each maximal ideal $P$.
\end{enumerate}
Then, $S^{-1}M$ has finite Goldie dimension for each $R$-module $M$ of finite Goldie dimension and for each multiplicative subset $S$ of $R$.
\end{proposition}
\begin{proof}
If $\mathrm{Gd}\ M<\infty$ then $M$ is a submodule of a finite direct sum of indecomposable injective modules $(E_i)_{1\leq i\leq n}$. It follows that $\mathrm{Gd}\ S^{-1}M<\infty$ if and only if $\mathrm{Gd}\ S^{-1}E_i<\infty$ for $i=1,\dots,n$. So, we may assume that $M$ is injective and indecomposable. On the other hand, since $\mathrm{End}_R(M)$ is a local ring, there exists a maximal ideal $P$ such that $M$ is a  module over $R_P$. So, we may assume that $R$ is local of maximal ideal $P$.

If $R$ satisfies $(1)$ then $S=R\setminus \{0\}$. Either $M$ is torsion-free and $S^{-1}M=M$, or $M$ is torsion and $S^{-1}M=0$.

If $R$ satisfies $(2)$, we may assume that $M=E(R/P)$ and $S\cap P\ne\emptyset$. Let $\phi$ be the natural map $M\rightarrow S^{-1}M$. Since $M$ is artinian by  \cite[Corollary 3.4]{Mat58} then so is the image of $\phi$. It follows that $S^{-1}M$ is an essential extension of a semisimple module $X$. But, for each $s\in S\cap P$, $sX=0$. We conclude that $S^{-1}M=0$.
\end{proof}

\begin{proposition}
\label{P:dGol} Let $R$ be a ring of Krull dimension zero. Then $\mathrm{dG}\ S^{-1}M<\infty$ for each module $M$ with $\mathrm{dG}\ M<\infty$ and for each multiplicative subset $S$ of $R$.
\end{proposition}
\begin{proof}
Since the natural maps $R\rightarrow S^{-1}R$ and $M\rightarrow S^{-1}M$ are surjective then $\mathrm{dG}\ S^{-1}M\leq\mathrm{dG}\ M$.
\end{proof}

\begin{example}
Let $R$ be a  local UFD of Krull dimension two, $p$ a prime element of $R$ and $S=\{p^n\mid n\in\mathbb{N}\}$. Then $\mathrm{dG}\ R=1$ and $\mathrm{dG}\ S^{-1}R=\infty$.
\end{example}
\begin{proof}
The first equality is obvious. Let $\Phi$ be the set of prime elements of $R$. If $P$ is the maximal ideal of $R$ then $P=\cup_{q\in\Phi} Rq$. So, $\Phi$ is not finite, else, by a classical lemma $P=Rq$ for some $q\in\Phi$ that is impossible. Let $n$ be a positive integer, let $q_1,\dots,q_n$ be $n$ distinct   elements  of $\Phi\setminus \{p\}$ and let $a=q_1\times\dots\times q_n$. By the chinese remainder theorem  $S^{-1}R/aS^{-1}R\cong S^{-1}R/q_1S^{-1}R\times\dots\times S^{-1}R/q_nS^{-1}R$. So, $\mathrm{dG}\ S^{-1}R\geq n$ for each $n>0$.
\end{proof}

\begin{theorem}
\label{T:main2} Let $R$ be a chain ring. The following conditions are equivalent:
\begin{enumerate}
\item For each  module $M$ of finite Goldie dimension and for each prime ideal $L$, $M_L$ has finite Goldie dimension;
\item for each prime ideal $L\ne 0$, $d_R(L)$ is finite.
\end{enumerate}
\end{theorem}
\begin{proof}
$(1)\Rightarrow (2)$.
By way of contradiction suppose there exists a non-zero prime ideal $L$ with $d_R(L)=\infty$. Then, for each integer $n>0$, $\widehat{R}/L\widehat{R}$ contains a torsion-free $R/L$- module $F$ of rank $n$. Let $A$ be an  ideal such that $A^{\sharp}=P$ and $A\subset L$. By \cite[Proposition 1.(2)]{Couc06} $\widehat{R}/A\widehat{R}$ is isomorphic to a submodule of the injective hull $E$ of $R/A$. Since $F_L$ is a $(R/L)_L$-vector space of dimension $n$ contained in $(\widehat{R}/L\widehat{R})_L$, we deduce from \cite[Proposition 21]{Coucho07}, applied to $(\widehat{R}/A\widehat{R})_L$, that $E_L$ contains a free $(R/A)_L$-module of rank $n$. So, $\mathrm{Gd}\ E_L\geq n$ for each integer $n$.

$(2)\Rightarrow (1)$. It is sufficient to show that $\mathrm{Gd}\ E_L<\infty$ for each  indecomposable injective module $E$ and each non-zero prime ideal $L$. Let $J=E_{\sharp}$. If $J\subseteq L$, then $E$ is a module over $R_J$, whence $E_L=E$. If $L\subset J$, since $d_{R_J}(L_J)\leq d_R(L)$, after replacing $R_J$ by $R$, we may assume that $J=P$. If $L=0$ (in the case where $R$ is a domain) then $E_L=0$. By \cite[Corollary 28]{Couch03}   $E$ is either faithful or annihilated by a simple ideal. So, if $L\ne 0$, there exists $A\in\mathcal{A}(E)$ such that $A\subset L$. We put $E'=(0:_EA)$. First we show that $E'_L$ is essential in $E_L$. Let $x\in E$ such that $\dfrac{x}{1}\notin E'_L$. Then $x\notin E'$. So, $(0:x)=rA$ where $r\in P$. It follows that $(0:rx)=A$. We conclude that $r\dfrac{x}{1}\in E'_L$ and $r\dfrac{x}{1}\ne 0$. Let $d=d_R(L)$. Then $(\widehat{R})_L/L(\widehat{R})_L\cong (R/L)_L^d$. Since $P/A$ is the set of zero-divisors of $R/A$ then $(R/A)_L$ is self FP-injective by \cite[Theorem 11(2)]{Couch03}. From \cite[Proposition 1]{Couc06} it is easy to deduce that $\widehat{R/A}$ is  singly projective over $R/A$. By \cite[Proposition 6]{Coucho07}  $(\widehat{R/A})_L$ singly projective over $(R/A)_L$. By \cite[Propositions 24 and 21]{Coucho07} $E'_L$ (which is isomorphic to $(\widehat{R/A})_L$) contains an essential free $(R/A)_L$-submodule of rank $d$.  We conclude that $\mathrm{Gd}\ E_L=d<\infty$.
\end{proof}

\begin{corollary}
\label{C:main2} Let $R$ be a chain ring and let $J$ be the intersection of all non-zero prime ideals. Then $J$ is  prime  ($=N$ if $R$ is not a domain) and the following assertions hold:
\begin{enumerate}
\item if $J\ne0$ then Goldie dimension finiteness is preserved by localization if and only if $d_R(J)$ is finite;
\item  if $J=0$ and if $0$ is  a non countable intersection of non-zero prime ideals then the finiteness of Goldie dimension  is preserved by localization if and only if there exists a non-zero prime ideal $L$ such that $d_R(L)$ is finite and $R_L$ is almost maximal.
\end{enumerate}
\end{corollary} 
\begin{proof}
 $(1)$ is an immediate consequence of Theorem~\ref{T:main2}.

$(2)$. First we will show that there exists a positive integer $p$ such that $d_R(L)\leq p$ for each non-zero prime ideal $L$. By way of contradiction suppose   there exists a non-zero prime ideal $L_n$ such that $d_R(L_n)\geq n$, for each integer $n>0$. Let $H=\cap_{n>0}L_n$. Then $H$ is a non-zero prime ideal and $d_R(H)\geq n$ for each integer $n>0$. We get a contradiction by Theorem~\ref{T:main2}. Let $p$ be the maximum of $d_R(I)$ where $I$ runs over all non-zero prime ideals of $R$ and let $L$ be the maximal prime ideal for which $d_R(L)=p$. If $L'$ is a non-zero prime ideal, $L'\subset L$, then $\mathrm{Mr}\ \widehat{R/L'}=p$. By Theorem~\ref{T:InjHull} $R_L/L'$ is maximal. We conclude that $R_L$ is almost maximal.
\end{proof}

Let us  observe that the following conditions:
\begin{enumerate}
\item each indecomposable injective $R$-module is polyserial;
\item the finiteness of Goldie dimension  is preserved by localization;
\end{enumerate}
are equivalent if $R$ is a valuation domain such that $0$ is a non countable intersection of non-zero prime ideals. But, generally these two conditions are not equivalent. For instance, if $J\ne 0$, $d_R(J)<\infty$ and $R_J$ not almost maximal, where $J$ is the intersection of all non-zero prime ideals, then $R$ satisfies condition $(2)$ but not condition $(1)$. Another example of a chain ring satisfying condition $(2)$ but not condition $(1)$ is the following:
\begin{example}
Let $R$ be a strongly discrete valuation domain whose set of non-zero prime ideals is $\{L_n\mid n\in\mathbb{N}\}$ with $L_0=P$ and $L_{n+1}\subset L_n$ for each $n\in\mathbb{N}$. Moreover we assume that $c_R(L_n)=p$ for each $0\ne n\in\mathbb{N}$, where $p$ is a prime integer. Such a  ring $R$ exists by \cite[Theorem 8]{FaZa86}. For each integer $n>0$, $d_R(L_n)=p^n$. So, condition $(2)$ is satisfied by Theorem~\ref{T:main2}. But condition $(1)$ doesn't hold because there is no non-zero prime ideal $L$ with $R_L$ almost maximal.
\end{example}

\end{document}